\documentclass[12pt, reqno]{amsart}
\usepackage{amsmath, amsthm, amscd, amsfonts, amssymb, graphicx, color}
\usepackage[bookmarksnumbered, colorlinks, plainpages]{hyperref}
\hypersetup{colorlinks=true,linkcolor=red, anchorcolor=green,
citecolor=cyan, urlcolor=red, filecolor=magenta, pdftoolbar=true}

\newtheorem{theorem}{Theorem}[section]
\newtheorem{lemma}[theorem]{Lemma}

\newtheorem{corollary}[theorem]{Corollary}
\theoremstyle{definition}
\newtheorem{definition}[theorem]{Definition}
\newtheorem{example}[theorem]{Example}

\theoremstyle{remark}
\newtheorem{remark}[theorem]{Remark}
\numberwithin{equation}{section}

  \DeclareMathOperator{\spe}{sp}
\begin{document}
\setcounter{page}{1}

\title[Pompeiu--\v{C}eby\v{s}ev type inequalities in inner product spaces]
{Pompeiu--\v{C}eby\v{s}ev type inequalities for selfadjoint
operators in Hilbert spaces}

\author[M. W.  Alomari]{Mohammad W. Alomari}

\address{Department of Mathematics, Faculty of Science and
Information Technology, Irbid National University, P.O. Box 2600,
Irbid, P.C. 21110, Jordan.}
\email{\textcolor[rgb]{0.00,0.00,0.84}{mwomath@gmail.com}}





\begin{abstract}
In this work, generalizations of some inequalities for continuous
$h$-synchronous ($h$-asynchronous) functions of selfadjoint linear
operators in Hilbert spaces are proved.
\end{abstract}

\maketitle

\section{Introduction}

Let $\mathcal{B}\left( H\right) $ be the Banach algebra of all
bounded linear operators defined on a complex Hilbert space
$\left( H;\left\langle \cdot ,\cdot \right\rangle \right)$  with
the identity operator  $1_H$ in $\mathcal{B}\left( H\right) $. Let
$A\in \mathcal{B}\left( H\right) $ be a selfadjoint linear
operator on $\left( H;\left\langle \cdot ,\cdot \right\rangle
\right)$. Let $C\left(\spe\left(A\right)\right)$ be the set of all
continuous functions defined on the spectrum of $A$
$\left(\spe\left(A\right)\right)$ and let $C^*\left(A\right)$ be
the $C^*$-algebra generated by $A$ and the identity operator
$1_H$.

Let us define the map $\mathcal{G}:
C\left(\spe\left(A\right)\right) \to C^*\left(A\right)$ with the
following properties (\cite{TF}, p.3):
\begin{enumerate}
\item $\mathcal{G}\left(\alpha f + \beta g\right) = \alpha
\mathcal{G}\left(f\right)+\beta \mathcal{G}\left(g\right)$, for
all scalars $\alpha, \beta$.

\item $\mathcal{G}\left(fg\right) = \mathcal{G}\left(f\right)
\mathcal{G}\left(g\right)$ and
$\mathcal{G}\left(\overline{f}\right)=\mathcal{G}\left(f\right)^*$;
where $\overline{f}$ denotes to the conjugate of $f$ and
$\mathcal{G}\left(f\right)^*$ denotes to the Hermitian of
$\mathcal{G}\left(f\right)$.

\item $\left\|\mathcal{G}\left(f\right)\right\|=\left\|f \right\|
= \mathop {\sup }\limits_{t \in \spe\left(A\right)} \left|
{f\left( t \right)} \right| $.

\item $\mathcal{G}\left( {f_0 } \right) = 1_H$ and
$\mathcal{G}\left( {f_1 } \right) = A$, where
$f_0\left(t\right)=1$ and $f_1\left(t\right)=t$ for all $t \in
\spe\left(A\right)$.
\end{enumerate}
Accordingly,  we define the continuous functional calculus for a
selfadjoint operator $A$ by
\begin{align*}
f\left(A\right) = \mathcal{G}\left(f\right)  \text{for all} \,f\in
C\left(\spe\left(A\right)\right).
\end{align*}
If both $f$ and $g$ are real valued functions on $\spe(A)$ then
the following important property holds:
\begin{align}
f\left( t \right) \ge g\left( t \right)  \,\text{for all} \, \,t
\in \spe\left( A \right) \,\,\text{implies}\,\, f\left( A \right)
\ge g\left( A \right), \label{eq1.2}
\end{align}
in the operator order of $\mathcal{B}(H)$.

In \cite{SD1}, Dragomir studied the \v{C}eby\v{s}ev functional
\begin{align}
\label{cebysev} C\left(f,g;A,x\right):= \left\langle {f\left( A
\right)g\left( A \right)x,x} \right\rangle - \left\langle {
g\left( A \right)x,x} \right\rangle \left\langle { f\left( A
\right)x,x} \right\rangle,
\end{align}
for any selfadjoint operator $A\in \mathcal{B}(H)$ and $x\in H$
with $\|x\|=1$.

To study the positivity of \eqref{cebysev}, Dragomir \cite{SD1}
introduced the following two results concerning continuous
synchronous (asynchronous) functions of selfadjoint linear
operators in Hilbert spaces.
\begin{theorem}
\label{thm1.1}Let $A$ be a selfadjoint operator with
$\spe\left(A\right)\subset \left[\gamma,\Gamma\right]$  for some
real numbers $\gamma,\Gamma$ with $\gamma<\Gamma$. If $f,g: \left[
{\gamma,\Gamma} \right]\to \mathbb{R}$ are continuous and
synchronous (asynchronous) on $\left[ {\gamma,\Gamma} \right]$,
then
\begin{align}
\label{eq1.3} \left\langle {f\left( A \right)g\left( A \right)x,x}
\right\rangle \ge (\le) \left\langle { g\left( A \right)x,x}
\right\rangle \left\langle { f\left( A \right)x,x} \right\rangle
\end{align}
for any $x\in H$ with $\|x\|=1$.
\end{theorem}

\begin{theorem}
\label{thm1.2} Let $A$ be a selfadjoint operator with
$\spe\left(A\right)\subset \left[\gamma,\Gamma\right]$  for some
real numbers $\gamma,\Gamma$ with $\gamma<\Gamma$.
\begin{enumerate}
\item If $f,g: \left[ {\gamma,\Gamma} \right]\to \mathbb{R}$ are
continuous and synchronous on $\left[ {\gamma,\Gamma} \right]$,
then
\begin{multline}
 \left\langle {f\left( A \right)g\left( A \right)x,x} \right\rangle  - \left\langle { f\left( A \right)x,x} \right\rangle  \cdot \left\langle { g\left( A \right)x,x} \right\rangle  \\
  \ge \left[ { \left\langle {f\left( A \right)x,x} \right\rangle  - f\left( {\left\langle {Ax,x} \right\rangle } \right) } \right] \\
\times \left[ {g\left( {\left\langle {Ax,x} \right\rangle }
\right)   -
 \left\langle
{g\left( A \right)x,x} \right\rangle } \right]  \label{eq1.4}
\end{multline}
 for any $x\in H$ with $\|x\|=1$.

 \item If $f,g: \left[ {\gamma,\Gamma} \right]\to \mathbb{R}$
are continuous and asynchronous on $\left[ {\gamma,\Gamma}
\right]$, then
\begin{multline}
\left\langle { f\left( A \right)x,x} \right\rangle \cdot
\left\langle { g\left( A \right)x,x} \right\rangle
-  \left\langle {f\left( A \right)g\left( A \right)x,x} \right\rangle  \\
  \ge \left[ { \left\langle {f\left( A \right)x,x} \right\rangle  - f\left( {\left\langle {Ax,x} \right\rangle } \right)
   } \right] \\
\times \left[ { \left\langle {g\left( A \right)x,x}
\right\rangle-g\left( {\left\langle {Ax,x} \right\rangle } \right)
} \right]  \label{eq1.5}
\end{multline}
 for any $x\in H$ with $\|x\|=1$.
\end{enumerate}
 \end{theorem}
For more related results, we refer the reader to \cite{SD2}, \cite{MB} and \cite{MM}.\\

Let $a,b\in \mathbb{R}$, $a<b$. Let $f,g,h:[a,b]\to \mathbb{R}$ be
three integrable functions, the Pompeiu--\v{C}eby\v{s}ev
functional was introduced in \cite{MA} such as:
\begin{align}
\label{Pompeiu.Chebyshev.Identity}\widehat{\mathcal{P}}_h\left(f,g\right)
 =\int_a^b {h^2 \left( t \right)dt} \int_a^b {f\left( t
\right)g\left( t \right)dt}   - \int_a^b {f\left( t \right)h\left(
t \right)dt} \int_a^b {h\left( t \right)g\left( t \right)dt}.
\end{align}
If we consider $h\left( x \right)=1$, then
\begin{align*}
\widehat{\mathcal{P}}_1\left(f,g\right)= \left(b-a\right)\int_a^b
{f\left( t \right)g\left( t \right)dt} - \int_a^b {f\left( t
\right)dt} \int_a^b {g\left( t \right)dt}=\left( b-a
\right)^2\mathcal{T}\left(f,g\right),
\end{align*}
which is the celebrated \v{C}eby\v{s}ev functional.\\

The corresponding version of  Pompeiu--\v{C}eby\v{s}ev functional
\eqref{Pompeiu.Chebyshev.Identity} for continuous functions of
selfadjoint linear operators in Hilbert spaces can be formulated
such as:
\begin{multline}
\label{P} \mathcal{P}\left(f,g,h;A,x\right):= \left\langle {h^2
\left( A \right)x,x} \right\rangle \left\langle {f\left( A
\right)g\left( A \right)x,x} \right\rangle
\\
-\left\langle {h\left( A \right)g\left( A \right)x,x}
\right\rangle \left\langle {h\left( A \right)f\left( A \right)x,x}
\right\rangle
\end{multline}
for $x\in H$ with $\|x\|=1$. This naturally, generalizes the
\v{C}eby\v{s}ev functional \eqref{cebysev}.

In this work, we introduce the $h$-synchronous ($h$-asynchronous)
where $h:\left[\gamma,\Gamma\right]\to \mathbb{R}_+$ is a
nonnegative function defined on $\left[\gamma,\Gamma\right]$
 for some real numbers $\gamma<\Gamma$. Accordingly, some inequalities for continuous $h$-synchronous ($h$-asynchronous)
functions of selfadjoint linear operators in Hilbert spaces of the
Pompeiu--\v{C}eby\v{s}ev functional \eqref{P} are proved. The
proof Techniques are similar to that ones used in \cite{SD2}.

\section{Main results}

 In \cite{MA}, the author  of this paper generalized the concept
 of monotonicity   as follows:
\begin{definition}
A real valued function $f$ defined on $\left[a,b\right]$ is said
to be increasing (decreasing) with respect to a positive function
$h:[a,b]\to \mathbb{R}_+$ or simply $h$-increasing
($h$-decreasing) if and only if
\begin{align*}
h\left( x \right)f\left( t \right) - h\left( t \right)f\left( x
\right) \ge (\le)\,\, 0,
\end{align*}
whenever $t \ge  x$  for every $x,t \in [a,b]$. In special case if
$h(x)=1$ we refer to the original monotonicity. Accordingly, for
$0<a<b$ we say that $f$ is $t^r$-increasing ($t^r$-decreasing) for
$r\in \mathbb{R} $ if and only if
\begin{align*}
x \le t \Longrightarrow x^r f\left( t \right)  - t^r f\left( x
\right) \ge (\le)\,\, 0
\end{align*}
for every $x,t \in [a,b]$.
\end{definition}
\begin{example}
Let $0<a<b$ and define $f :\left[a,b\right]\to \mathbb{R}$  given
by
\begin{enumerate}
\item $f(s)=1$, then $f$ is $t^r$-decreasing for all $r>0$ and
$t^r$-increasing for all $r<0$.

\item $f(s)=s$, then $f$ is $t^r$-decreasing for all $r>1$ and
$t^r$-increasing for all $r<1$.

\item $f(s)=s^{-1}$, then $f$ is $t^r$-decreasing for all $r>-1$
and $t^r$-increasing for all $r<-1$.

\end{enumerate}
\end{example}
\begin{lemma}
\label{lemma1}Every $h$-increasing function is increasing. The
converse need not be true.
\end{lemma}

\begin{proof}
If $h=0$ nothing to prove. For $h\ne0$, if $f$ is $h$-increasing
on $\left[a,b\right]$, then
\begin{align*}
x \le t \Longrightarrow 0\le h\left( x \right)f\left( t
\right)-h\left( t \right)f\left( x \right)\le h\left( t \right)
\left(f\left( t \right)-f\left( x \right)\right) \Longrightarrow
f\left( x \right)\le f\left( t \right),
\end{align*}
which means that $f$ increases on $\left[a,b\right]$.
\end{proof}

There exists $h$-increasing ($h$-decreasing) function which is not
increasing (decreasing). For example, consider the function
$f:(0,1)\to \mathbb{R}$,  given by $f(s)=s(1-s)$, $0< s <1$.
Clearly,   $f(s)$ is increasing on $(0,1/2)$ and decreasing on
$(1/2,1)$. While  if $1> t\ge x > 0$, then
\begin{align*}
 xt(1-t)-tx(1-x) =  xt (x-t) \le 0,
\end{align*}
i.e., $f$ is $t$-decreasing on $(0,1)$.   As a special case of
Lemma \ref{lemma1}, for $a,b\in \mathbb{R}$, $0<a<b$ and a
positive function $h:[a,b]\to \mathbb{R}_+$,   if
$f:\left[a,b\right]\to \mathbb{R} $ is $t^r$-increasing for $r>0$
($t^r$-decreasing for $r<0$), then $f$ is increasing (decreasing)
on $\left[a,b\right]$.

The concept of synchronization  has a wide range of usage in
several areas of mathematics. Simply, two functions
$f,g:\left[a,b\right]\to \mathbb{R}$ are called synchronous
(asynchronous)  if and only if the inequality
\begin{align*}
\left( { f\left( t \right) - f\left( x \right)} \right)\left( {
g\left( t \right) -  g\left( x\right)} \right) \ge (\le)\,\, 0,
\end{align*}
holds  for all $x,t\in \left[a,b\right]$.

Next, we define the concept of $h$-synchronous ($h$-asynchronous)
functions.
\begin{definition}
\label{def2}The real valued functions $f,g:\left[a,b\right]\to
\mathbb{R}$ are called   synchronous (asynchronous) with respect
to a non-negative function $h:[a,b]\to \mathbb{R}_+$ or simply
$h$-synchronous ($h$-asynchronous)  if and only if
\begin{align}
\label{h-syn}\left( {h\left( y \right)f\left( x \right) - h\left(
x \right)f\left( y \right)} \right)\left( {h\left( y
\right)g\left( x \right) - h\left( x \right)g\left( y \right)}
\right) \ge (\le)\,\, 0
\end{align}
for all $x,y \in \left[a,b\right]$.

In other words if both $f$ and $g$ are either $h$-increasing or
$h$-decreasing then $$\left( {h\left( y \right)f\left( x \right) -
h\left( x \right)f\left( y \right)} \right)\left( {h\left( y
\right)g\left( x \right) - h\left( x \right)g\left( y \right)}
\right) \ge0.$$ While, if one of the function is $h$-increasing
and the other is $h$-decreasing then $$\left( {h\left( y
\right)f\left( x \right) - h\left( x \right)f\left( y \right)}
\right)\left( {h\left( y \right)g\left( x \right) - h\left( x
\right)g\left( y \right)} \right) \le0.$$

In special case if $h(x)=1$ we refer to the original
synchronization. Accordingly, for $0<a<b$ we say that $f$ and $g$
are $t^r$-synchronous ($t^r$-asynchronous) for $r\in \mathbb{R} $
if and only if
\begin{align*}
 \left(x^r f\left( t \right)  - t^r f\left(
x \right) \right) \left(x^r g\left( t \right)  - t^r g\left( x
\right) \right)\ge (\le)\,\, 0
\end{align*}
for every $x,t \in [a,b]$.
\end{definition}

\begin{remark}
In Definition \eqref{def2}, if $f=g$ then $f$ and $g$   are always
$h$-synchronous regardless of $h$-monotonicity of $f$ (or $g$). In
other words, a function $f$ is always $h$-synchronous with itself.
\end{remark}

\begin{example}
Let $0<a<b$ and define $f,g :\left[a,b\right]\to \mathbb{R}$ given
by
\begin{enumerate}
\item $f(s)=1=g(s)$, then $f$ and $g$ are $t^r$-synchronous for
all $r\in \mathbb{R}$.

\item $f(s)=1$ and $g(s)=s$, then $f$ is $t^r$-synchronous
 for all $r \in \left( { - \infty ,0} \right) \cup \left( {1,\infty } \right)$ and
$t^r$-asynchronous for all
 $0<r<1$.

\item  $f(s)=1$ and   $g(s)=s^{-1}$, then $f$ is $t^r$-synchronous
for all $r \in \left( { - \infty ,-1} \right) \cup \left(
{0,\infty } \right)$ and $t^r$-asynchronous for all $-1<r<0$.

\item $f(s)=s$ and   $g(s)=s^{-1}$, then $f$ is $t^r$-synchronous
for all $r \in \left( { - \infty ,-1} \right) \cup \left(
{1,\infty } \right)$ and $t^r$-asynchronous for all $-1<r<1$.
\end{enumerate}
\end{example}

Let us start with the following result regarding the positivity of
$\mathcal{P}\left(f,g,h;A,x\right)$.
\begin{theorem}
\label{thm2.1}Let $A$ be a selfadjoint operator with
$\spe\left(A\right)\subset \left[\gamma,\Gamma\right]$  for some
real numbers $\gamma,\Gamma$ with $\gamma<\Gamma$. Let
$h:\left[\gamma,\Gamma\right]\to \mathbb{R}_+$ be a non-negative
 and continuous function. If $f,g:
\left[ {\gamma,\Gamma} \right]\to \mathbb{R}$ are continuous and
both $f$ and $g$ are $h$-synchronous ($h$-asynchronous) on $\left[
{\gamma,\Gamma} \right]$, then
\begin{align}
\label{eq2.1} \left\langle {h^2 \left( A \right)x,x} \right\rangle
\left\langle {f\left( A \right)g\left( A \right)x,x} \right\rangle
\ge (\le) \left\langle {h\left( A \right)g\left( A \right)x,x}
\right\rangle \left\langle {h\left( A \right)f\left( A \right)x,x}
\right\rangle
\end{align}
for any $x\in H$ with $\|x\|=1$.
\end{theorem}

\begin{proof}
Since $f$ and $g$ are $h$-synchronous then
\begin{align*}
\left( {h\left( s \right)f\left( t \right) - h\left( t
\right)f\left( s \right)} \right)\left( {h\left( s \right)g\left(
t \right) - h\left( t \right)g\left( s \right)} \right)  \ge 0,
\end{align*}
and this is allow us to write
\begin{align}
h^2\left( s \right)f\left( t \right)g\left( t \right)+h^2\left( t
\right)f\left( s \right)g\left( s \right) \ge h\left( s
\right)h\left( t \right)f\left( t \right)g\left( s \right)
+h\left( s \right)h\left( t \right) g\left( t \right)f\left( s
\right) \label{eq2.2}
\end{align}
for all $t, s\in [a,b]$. We fix $s \in \left[a,b\right]$ and apply
property \eqref{eq1.2} for inequality \eqref{eq2.2}, then we have
for each $x \in H$ with $\|x\|=1$, that
\begin{multline*}
\left\langle {\left( {h^2 \left( s \right)f\left( A \right)g\left(
A \right) + h^2 \left( A \right)f\left( s \right)g\left( s
\right)} \right)x,x} \right\rangle
\\
\ge \left\langle { \left({h\left( A \right)f\left( A
\right)h\left( s \right)g\left( s \right) + h\left( A
\right)g\left( A \right)h\left( s \right)f\left( s
\right)}\right)x,x} \right\rangle,
\end{multline*}
and this equivalent to write
\begin{multline}
 h^2 \left( s \right)\left\langle {f\left( A \right)g\left( A \right)x,x} \right\rangle  + f\left( s \right)g\left( s \right)\left\langle {h^2 \left( A \right)x,x} \right\rangle  \\
  \ge h\left( s \right)g\left( s \right)\left\langle {h\left( A \right)f\left( A \right)x,x} \right\rangle  + h\left( s \right)f\left( s \right)\left\langle {h\left( A \right)g\left( A \right)x,x}
  \right\rangle.\label{eq2.3}
\end{multline}
Applying property \eqref{eq1.2} again for inequality
\eqref{eq2.3}, then we have for each $y\in H$ with $\|y\|=1$, that
\begin{multline*}
 \left\langle {\left( {h^2 \left( A \right)\left\langle {f\left( A \right)g\left( A \right)x,x} \right\rangle  + f\left( A \right)g\left( A \right)\left\langle {h^2 \left( A \right)x,x} \right\rangle } \right)y,y} \right\rangle  \\
  \ge \left\langle {\left( {h\left( A \right)g\left( A \right)\left\langle {h\left( A \right)f\left( A \right)x,x} \right\rangle  + h\left( A\right)f\left( A \right)\left\langle {h\left( A \right)g\left( A \right)x,x} \right\rangle } \right)y,y}
  \right\rangle,
\end{multline*}
which gives
\begin{multline}
 \left\langle {h^2 \left( A \right)y,y} \right\rangle \left\langle {f\left( A \right)g\left( A \right)x,x} \right\rangle  + \left\langle {h^2 \left( A \right)x,x} \right\rangle \left\langle {f\left( A \right)g\left( A \right)y,y} \right\rangle  \\
  \ge \left\langle {h\left( A \right)g\left( A \right)y,y} \right\rangle \left\langle {h\left( A \right)f\left( A \right)x,x} \right\rangle  + \left\langle {h\left( A \right)g\left( A \right)x,x} \right\rangle \left\langle {h\left( A \right)f\left( A \right)y,y}
  \right\rangle  \label{eq2.4}
\end{multline}
 for each $x,y\in H$ with $\|x\|=\|y\|=1$, which gives more than we need, so that by setting $y=x$ in \eqref{eq2.4} we get the `$\ge$' case in
 \eqref{eq2.1}. The revers case follows trivially, and this
 completes the proof.
\end{proof}

\begin{corollary}
\label{cor1}Let $A$ be a selfadjoint operator with
$\spe\left(A\right)\subset \left[\gamma,\Gamma\right]$  for some
real numbers $\gamma,\Gamma$ with $\gamma<\Gamma$. Let
$h:\left[\gamma,\Gamma\right]\to \mathbb{R}_+$ be a non-negative
 and continuous function. If $f: \left[
{\gamma,\Gamma} \right]\to \mathbb{R}$ is continuous and
$h$-synchronous  on $\left[ {\gamma,\Gamma} \right]$, then
\begin{align}
\label{ineq.cor1}\left\langle {h\left( A \right)f\left( A
\right)x,x} \right\rangle^2 \le  \left\langle {h^2 \left( A
\right)x,x} \right\rangle \left\langle {f^2\left( A \right)x,x}
\right\rangle
\end{align}
for each $x \in H$ with $\|x\|=1$.
\end{corollary}
\begin{proof}
Setting $f=g$ in \eqref{eq2.1}  we get the desired result.
\end{proof}

\begin{remark}
It is easy to check that the function $f(s)=s^{1/2}$ is
$t^{-1/2}$-synchronous for all $s,t>0$. Applying \eqref{ineq.cor1}
the for  $0<\gamma<\Gamma$ we get that
\begin{align*}
1 \le  \left\langle {A^{-1}x,x} \right\rangle \left\langle {Ax,x}
\right\rangle.
\end{align*}
Also, since $\gamma\cdot 1_H\le A \le \Gamma \cdot 1_H$, then the
Kanotrovich inequality reads
\begin{align*}
 \left\langle {A^{-1}x,x} \right\rangle \left\langle {Ax,x}
\right\rangle \le
\frac{\left(\gamma+\Gamma\right)^2}{4\gamma\Gamma},
\end{align*}
combining the above two inequalities we get
\begin{align}
1\le \left\langle {A^{-1}x,x} \right\rangle \left\langle {Ax,x}
\right\rangle \le
\frac{\left(\gamma+\Gamma\right)^2}{4\gamma\Gamma}.\label{refin.kant.ineq}
\end{align}
\end{remark}

\begin{corollary}
\label{cor2} Let $A$ be a selfadjoint operator with
$\spe\left(A\right)\subset \left[\gamma,\Gamma\right]$  for some
real numbers $\gamma,\Gamma$ with $0<\gamma<\Gamma$. If $f,g:
\left[ {\gamma,\Gamma} \right]\to \mathbb{R}$ are continuous and
$t$-synchronous ($t$-asynchronous) on $\left[ {\gamma,\Gamma}
\right]$, then
\begin{align}
\label{eq3.2.Alomari2017} \left\langle {A^2 x,x} \right\rangle
\left\langle {f\left( A \right)g\left( A \right)x,x} \right\rangle
\ge (\le) \left\langle {Ag\left( A \right)x,x} \right\rangle
\left\langle {Af\left( A \right)x,x} \right\rangle
\end{align}
for each $x \in H$ with $\|x\|=1$.
\end{corollary}
\begin{proof}
Setting $h(t)=t$ in \eqref{eq2.1} we get the desired result.
\end{proof}

Before we state our next remark, we interested to give  the
following example.
\begin{example}
\label{example2.11}
\begin{enumerate}
\item If $f(s)=s^p$ and $g(s)=s^q$ ($s>0$), then $f$ and $g$ are
$t^r$-synchronous for all $p,q>r>0$  and $t^r$-asynchronous
 for all $p> r>q>0$.

\item If $f(s)=s^p$ and $g(s)=\log (s)$ ($s>1$), then  $f$ is
$t^r$-synchronous
 for all $p < r<0$ and $t^r$-asynchronous
 for all $r< p<0$.

\item If  $f(s)=\exp(s)=g(s)$, then $f$ is $t^r$-synchronous for
all for all $r\in \mathbb{R}$.
\end{enumerate}
\end{example}

\begin{remark}
\label{remark1}From the proof of the above theorem we observe,
that, if $A$ and $B$ are selfadjoint operators such that
$\spe\left(A\right),
\spe\left(B\right)\in\left[\gamma,\Gamma\right] $; and
$h:\left[\gamma,\Gamma\right] \to \mathbb{R}_+$ is non-negative
continuous, then for any continuous functions
$f,g:\left[\gamma,\Gamma\right] \to \mathbb{R}$  which are both
 $h$-synchronous ($h$-asynchronous)
\begin{multline}
 \left\langle {h^2 \left( B \right)y,y} \right\rangle \left\langle {f\left( A \right)g\left( A \right)x,x} \right\rangle  + \left\langle {h^2 \left( A \right)x,x} \right\rangle \left\langle {f\left( B \right)g\left( B \right)y,y} \right\rangle  \\
  \ge (\le) \left\langle {h\left( B \right)g\left(B \right)y,y} \right\rangle \left\langle {h\left( A \right)f\left( A \right)x,x} \right\rangle \\ + \left\langle {h\left( A \right)g\left( A \right)x,x} \right\rangle \left\langle {h\left( B \right)f\left( B \right)y,y}
  \right\rangle\label{eq2.6}
\end{multline}
for each $x,y\in H$ with $\|x\|=\|y\|=1$. Using Example
\ref{example2.11} we can observe the following special cases:
 \begin{enumerate}
 \item If $f\left( s \right)=s^p$ and $g\left( s \right)=s^q$ ($s>0$),
then $f$ and $g$ are $t^r$-synchronous for all $p,q > r>0$, so
that we have
\begin{multline*}
 \left\langle {B^{2r}y,y} \right\rangle \left\langle {A^{p+q}x,x} \right\rangle  + \left\langle {A^{2r}x,x} \right\rangle \left\langle {B^{p+q}y,y} \right\rangle  \\
  \ge   \left\langle {B^{q+r}y,y} \right\rangle \left\langle {A^{p+r}x,x} \right\rangle  + \left\langle {A^{q+r}x,x} \right\rangle \left\langle {B^{p+r} y,y}
  \right\rangle.
\end{multline*}
If $p>r>q>0$, then $f$ and $g$ are $t^r$-asynchronous and thus the
reverse inequality holds.\\

\item If $f\left( s \right)=s^p$   and $g\left( s \right)=\log
 s$ ($s>1$),
 then $f$ and $g$ are $t^r$-synchronous for all $p < r<0$ we have
\begin{multline*}
 \left\langle {B^{2r}y,y} \right\rangle \left\langle {A^p \log\left( A \right)x,x} \right\rangle  + \left\langle {A^{2r}x,x} \right\rangle \left\langle {B^p\log\left( B \right)y,y} \right\rangle  \\
  \ge   \left\langle {B^r\log\left(B \right)y,y} \right\rangle \left\langle {A^{p+r}x,x} \right\rangle  + \left\langle {A\log\left( A \right)x,x} \right\rangle \left\langle {B^{p+r}y,y}
  \right\rangle.
\end{multline*}
If $r<p<0$, then $f$ and $g$ are $t^r$-asynchronous and thus the reverse inequality holds.\\

 \item If $f\left( s \right)=\exp\left(s\right)=g\left( s \right)$,
 then $f$ and $g$ are $t^r$-synchronous for all $r \in
 \mathbb{R}$, so that we have
 \begin{multline*}
 \left\langle {B^{2r}y,y} \right\rangle \left\langle {\exp\left( 2A \right)x,x} \right\rangle  + \left\langle {A^{2r}x,x} \right\rangle \left\langle {\exp\left( 2B \right)y,y} \right\rangle  \\
  \ge   2 \left\langle {A^{r}\exp\left( A \right)x,x} \right\rangle\left\langle {B^{r}\exp\left(B \right)y,y} \right\rangle.\\
\end{multline*}
 \end{enumerate}
Therefore, by choosing an appropriate function $h$ such that the
assumptions in Remark \ref{remark1} are fulfilled then one may
generate family of inequalities from \eqref{eq2.6}.
\end{remark}

\begin{corollary}
\label{cor3} Let $A$ be a selfadjoint operator with
$\spe\left(A\right)\subset \left[\gamma,\Gamma\right]$   for some
real numbers $\gamma,\Gamma$ with $0<\gamma<\Gamma$. If $f: \left[
{\gamma,\Gamma} \right]\to \mathbb{R}$ is continuous and $f$ is
$t$-synchronous   on $\left[ {\gamma,\Gamma} \right]$, then
\begin{align}
 \left\langle {A^2 x,x} \right\rangle
\left\langle {f\left( A \right)x,x} \right\rangle \ge
 \left\langle {Ax,x} \right\rangle \left\langle {Af\left( A
\right)x,x} \right\rangle  \label{eq2.7}
\end{align}
for each $x \in H$ with $\|x\|=1$. In particular, if $f(s)=s^p$
$(p >1)$  for all $s \in \left[ {\gamma,\Gamma} \right]$, then
\begin{align*}
\left\langle {A^2x,x} \right\rangle \left\langle {A^px,x}
\right\rangle \ge  \left\langle {Ax,x} \right\rangle \left\langle
{A^{p+1}x,x} \right\rangle.
\end{align*}
\end{corollary}
\begin{proof}
Setting $f=g$ in Corollary \ref{cor2} we get the desired result.
\end{proof}

\begin{corollary}
\label{cor4} Let $A$ be a selfadjoint operator with
$\spe\left(A\right)\subset \left[\gamma,\Gamma\right]$  for some
real numbers $\gamma,\Gamma$ with $\gamma<\Gamma$. Let $h: \left[
{\gamma,\Gamma} \right]\to \mathbb{R}$ be a non-negative
continuous. If $f: \left[ {\gamma,\Gamma} \right]\to \mathbb{R}$
is continuous and  $h$-synchronous, then
\begin{align}
 \left\langle {h^2 \left( A \right)x,x} \right\rangle
\left\langle {f\left( A \right)x,x} \right\rangle \ge \left\langle
{h\left( A \right)x,x} \right\rangle \left\langle {h\left( A
\right)f\left( A \right)x,x} \right\rangle \label{eq2.9}
\end{align}
for each $x \in H$ with $\|x\|=1$. In particular, if $f(s)=s^p$
 is $h$-synchronous for all $s \in \left[ {\gamma,\Gamma} \right]$, then we
have
\begin{align*}
\left\langle {h^2 \left( A \right)x,x} \right\rangle \left\langle
{A^px,x} \right\rangle \ge  \left\langle {h\left( A \right)x,x}
\right\rangle \left\langle {h\left( A \right)A^px,x}
\right\rangle.
\end{align*}
 \end{corollary}

\begin{remark}
Setting $f(s)=s^{-1}$, $\forall s \in \left[ {\gamma,\Gamma}
\right]$ in \eqref{eq2.9} (in this case we assume
$0<\gamma<\Gamma$) then for each $x \in H$ with $\|x\|=1$, we have
\begin{align*}
 \left\langle {h^2 \left( A \right)x,x} \right\rangle
\ge  \left\langle {h\left( A \right) A^{-1} x,x} \right\rangle
\left\langle {Ah\left( A \right)x,x} \right\rangle,
\end{align*}
provided that $s^{-1}$ is $h$-synchronous on $\left[
{\gamma,\Gamma} \right]$.
\end{remark}

\begin{theorem}
\label{thm2.2} Let $A$ be a selfadjoint operator with
$\spe\left(A\right)\subset \left[\gamma,\Gamma\right]$ for some
real numbers $\gamma,\Gamma$ with $\gamma<\Gamma$. Let $h: \left[
{\gamma,\Gamma} \right]\to \mathbb{R}$ be a non-negative
continuous.
 If $f,g: \left[ {\gamma,\Gamma} \right]\to \mathbb{R}$ are
continuous and both $f$ and $g$ are $h$-synchronous
($h$-asynchronous) on $\left[ {\gamma,\Gamma} \right]$, then
\begin{multline}
h^2 \left( {\left\langle {Ax,x} \right\rangle }
\right)\left\langle {f\left( A \right)g\left( A \right)x,x}
\right\rangle  - \left\langle {h\left( A \right)f\left( A
\right)x,x} \right\rangle \cdot \left\langle {h\left( A
\right)g\left( A \right)x,x} \right\rangle
\\
 \ge (\le) \left[{h\left( {\left\langle {Ax,x} \right\rangle } \right)
\left\langle {h\left( A \right)f\left( A \right)x,x} \right\rangle
-   \left\langle {h^2 \left( A \right)x,x} \right\rangle f\left(
{\left\langle {Ax,x} \right\rangle } \right) }\right] \cdot
g\left( {\left\langle {Ax,x} \right\rangle } \right)
 \\
 + \left[{h\left( {\left\langle {Ax,x} \right\rangle }
\right)f\left( {\left\langle {Ax,x} \right\rangle } \right) -
\left\langle {h\left( A \right)f\left( A \right)x,x} \right\rangle
}\right] \cdot \left\langle {h\left( A \right)g\left( A
\right)x,x} \right\rangle  \label{eq2.10}
\end{multline}
 for any $x\in H$ with $\|x\|=1$.
\end{theorem}

\begin{proof}
Since $f, g$ are synchronous and $\gamma \le  \left\langle {Ax,x}
\right\rangle \le \Gamma$  for any $x\in H$ with $\|x\|=1$, we
have
\begin{multline}
\left( {h\left( \left\langle {Ax,x} \right\rangle \right)f\left( t
\right) - h\left( t \right)f\left( \left\langle {Ax,x}
\right\rangle \right)} \right)
\\
\times\left( {h\left( \left\langle {Ax,x} \right\rangle
\right)g\left( t \right) - h\left( t \right)g\left( \left\langle
{Ax,x} \right\rangle \right)} \right) \ge 0  \label{eq2.12}
\end{multline}
for any $t \in \left[a,b\right]$ for any $x\in H$ with $\|x\|=1$.

Employing property \eqref{eq1.2} for inequality \eqref{eq2.12} we
have
\begin{multline}
\left\langle{\left[{ h\left( \left\langle {Ax,x} \right\rangle
\right)f\left( B \right) - h\left( B \right)f\left( \left\langle
{Ax,x} \right\rangle \right) }\right] }\right.
\\
\left.{\times\left[{ h\left( \left\langle {Ax,x} \right\rangle
\right)g\left( B \right) - h\left( B \right)g\left( \left\langle
{Ax,x} \right\rangle \right) }\right]y,y}\right\rangle \ge 0
\label{eq2.13}
\end{multline}
for any  bounded linear operator $B$ with $\spe\left({B}\right)
\subseteq \left[\gamma,\Gamma\right]$ and $y\in H$ with $\|y\|=1$.

Now, since
\begin{multline}
 \left\langle {\left[ {h\left( {\left\langle {Ax,x} \right\rangle } \right)f\left( B \right) - h\left( B \right)f\left( {\left\langle {Ax,x} \right\rangle } \right)} \right]}\right.\\ \times \left.{\left[ {h\left( {\left\langle {Ax,x} \right\rangle } \right)g\left( B \right) - h\left( B \right)g\left( {\left\langle {Ax,x} \right\rangle } \right)} \right]y,y} \right\rangle  \\
  = h^2 \left( {\left\langle {Ax,x} \right\rangle } \right)\left\langle {f\left( B \right)g\left( B \right)y,y} \right\rangle  - h\left( {\left\langle {Ax,x} \right\rangle } \right)f\left( {\left\langle {Ax,x} \right\rangle } \right)\left\langle {h\left( B \right)g\left( B \right)y,y} \right\rangle  \\
  - h\left( {\left\langle {Ax,x} \right\rangle } \right)g\left( {\left\langle {Ax,x} \right\rangle } \right)\left\langle {h\left( B \right)f\left( B \right)y,y} \right\rangle
  \\+ \left\langle {h^2 \left( B \right)y,y} \right\rangle f\left( {\left\langle {Ax,x} \right\rangle } \right)g\left( {\left\langle {Ax,x} \right\rangle }
  \right),\label{eq2.14}
\end{multline}
then from \eqref{eq2.13} we get
\begin{multline}
 h^2 \left( {\left\langle {Ax,x} \right\rangle } \right)\left\langle {f\left( B \right)g\left( B \right)y,y} \right\rangle  + \left\langle {h^2 \left( B \right)y,y} \right\rangle f\left( {\left\langle {Ax,x} \right\rangle } \right)g\left( {\left\langle {Ax,x} \right\rangle } \right) \\
  \ge h\left( {\left\langle {Ax,x} \right\rangle } \right)f\left( {\left\langle {Ax,x} \right\rangle } \right)\left\langle {h\left( B \right)g\left( B \right)y,y} \right\rangle  \\+ h\left( {\left\langle {Ax,x} \right\rangle } \right)g\left( {\left\langle {Ax,x} \right\rangle } \right)\left\langle {h\left( B \right)f\left( B \right)y,y}
  \right\rangle,\label{eq2.15}
\end{multline}
and this is equivalent to write
\begin{align}
&h^2 \left( {\left\langle {Ax,x} \right\rangle }
\right)\left\langle {f\left( B \right)g\left( B \right)y,y}
\right\rangle- \left\langle {h\left( A \right)f\left( A
\right)x,x} \right\rangle \cdot \left\langle {h\left( A
\right)g\left( A \right)x,x} \right\rangle
\label{eq2.16}\\
&\ge g\left( {\left\langle {Ax,x} \right\rangle }
\right)\left[{h\left( {\left\langle {Ax,x} \right\rangle } \right)
\left\langle {h\left( B \right)f\left( B \right)y,y} \right\rangle
-   \left\langle {h^2 \left( B \right)y,y} \right\rangle f\left(
{\left\langle {Ax,x} \right\rangle } \right) }\right]
\nonumber\\
&\qquad+ h\left( {\left\langle {Ax,x} \right\rangle }
\right)f\left( {\left\langle {Ax,x} \right\rangle }
\right)\left\langle {h\left( B \right)g\left( B \right)y,y}
\right\rangle - \left\langle {h\left( A \right)f\left( A
\right)x,x} \right\rangle \cdot \left\langle {h\left( A
\right)g\left( A \right)x,x} \right\rangle\nonumber
\end{align}
for each $x,y\in H$ with $\|x\|=\|y\|=1$. Setting $B=A$ and $y=x$
in \eqref{eq2.16} we get the required result in \eqref{eq2.10}.
The reverse sense follows similarly.
 \end{proof}

 \begin{remark}
Let $0 < \gamma < \Gamma$ and choose $f(s)=s$ and $g(s)=s^{-1}$,
$s>0$ in Theorem \ref{thm2.2}. So that, if $f$ and $g$ are
$h$-synchronous ($h$-asynchronous) on $\left[ {\gamma,\Gamma}
\right]$, then
\begin{multline*}
h^2 \left( {\left\langle {Ax,x} \right\rangle } \right)  -
\left\langle {Ah\left( A \right) x,x} \right\rangle \cdot
\left\langle {A^{-1}h\left( A \right) x,x} \right\rangle
\\
 \ge (\le) \left[{h\left( {\left\langle {Ax,x} \right\rangle } \right)
\left\langle {Ah\left( A \right) x,x} \right\rangle - \left\langle
{h^2 \left( A \right)x,x} \right\rangle   \left\langle {Ax,x}
\right\rangle   }\right] \cdot  \left\langle {Ax,x} \right\rangle
^{-1}
 \\
 + \left[{h\left( {\left\langle {Ax,x} \right\rangle }
\right) \left\langle {Ax,x} \right\rangle   - \left\langle
{Ah\left( A \right)x,x} \right\rangle }\right] \cdot \left\langle
{A^{-1}h\left( A \right) x,x} \right\rangle
\end{multline*}
 for any $x\in H$ with $\|x\|=1$. In special case, if $h(t)=1$ for
 all $t \in \left[ {\gamma,\Gamma}
\right]$, then $s$ and $s^{-1}$ are  asynchronous so that we have
\begin{align}
\label{conv} 1 \le \left\langle {A x,x} \right\rangle \left\langle
{A^{-1} x,x} \right\rangle
\end{align}
 \end{remark}

 \begin{corollary}
\label{cor5}Let $A$ be a selfadjoint operator with
$\spe\left(A\right)\subset \left[\gamma,\Gamma\right]$  for some
real numbers $\gamma,\Gamma$ with $\gamma<\Gamma$. Let $h: \left[
{\gamma,\Gamma} \right]\to \mathbb{R}$ be a non-negative
continuous. If $f: \left[ {\gamma,\Gamma} \right]\to \mathbb{R}$
is continuous and  $h$-synchronous on $\left[ {\gamma,\Gamma}
\right]$, then
\begin{multline}
h^2 \left( {\left\langle {Ax,x} \right\rangle }
\right)\left\langle {f^2\left( A \right)x,x} \right\rangle  -
\left\langle {h\left( A \right)f\left( A \right)x,x}
\right\rangle^2
\\
 \ge   \left[{h\left( {\left\langle {Ax,x} \right\rangle } \right)
\left\langle {h\left( A \right)f\left( A \right)x,x} \right\rangle
-   \left\langle {h^2 \left( A \right)x,x} \right\rangle f\left(
{\left\langle {Ax,x} \right\rangle } \right) }\right] \cdot
f\left( {\left\langle {Ax,x} \right\rangle } \right)
 \\
 + \left[{h\left( {\left\langle {Ax,x} \right\rangle }
\right)f\left( {\left\langle {Ax,x} \right\rangle } \right) -
\left\langle {h\left( A \right)f\left( A \right)x,x} \right\rangle
}\right] \cdot \left\langle {h\left( A \right)f\left( A
\right)x,x} \right\rangle
  \label{eq2.17}
\end{multline}
 for any $x\in H$ with $\|x\|=1$.
\end{corollary}
\begin{proof}
Setting $f=g$ in \eqref{eq2.10}, respectively, we get the required
results.
\end{proof}

 \begin{corollary}
\label{cor6} Let $A$ be a selfadjoint operator with
$\spe\left(A\right)\subset \left[\gamma,\Gamma\right]$ for some
real numbers $\gamma,\Gamma$ with $0<\gamma<\Gamma$.
 If $f : \left[ {\gamma,\Gamma} \right]\to \mathbb{R}$ are
continuous and $t$-synchronous on $\left[ {\gamma,\Gamma}
\right]$, then
\begin{multline}
\left\langle {Ax,x} \right\rangle^2\left\langle {f^2\left( A
\right)x,x} \right\rangle  - \left\langle {Af\left( A \right)x,x}
\right\rangle^2
\\
 \ge   \left[{ \left\langle {Ax,x} \right\rangle
\left\langle {Af\left( A \right)x,x} \right\rangle - \left\langle
{A^2x,x} \right\rangle f\left( {\left\langle {Ax,x} \right\rangle
} \right) }\right] \cdot f\left( {\left\langle {Ax,x}
\right\rangle } \right)
 \\
 + \left[{ \left\langle {Ax,x} \right\rangle  f\left( {\left\langle {Ax,x} \right\rangle } \right) -
\left\langle {Af\left( A \right)x,x} \right\rangle }\right] \cdot
\left\langle {Af\left( A \right)x,x} \right\rangle
\end{multline}
 for any $x\in H$ with $\|x\|=1$.
 \end{corollary}
\begin{proof}
Setting $h(t)=t$ in \eqref{eq2.17}, respectively, we get the
required results.
\end{proof}

\begin{theorem}
\label{thm2.2-3} Let $A$ be a selfadjoint operator with
$\spe\left(A\right)\subset \left[\gamma,\Gamma\right]$ for some
real numbers $\gamma,\Gamma$ with $\gamma<\Gamma$. Let $h: \left[
{\gamma,\Gamma} \right]\to \mathbb{R}$ be a non-negative
continuous.
 If $f,g: \left[ {\gamma,\Gamma} \right]\to \mathbb{R}$ are
continuous and both $f$ and $g$ are $h$-synchronous
($h$-asynchronous) on $\left[ {\gamma,\Gamma} \right]$, then
\begin{align}
& h^2\left( \left\langle {Ax,x} \right\rangle \right)f\left(
\left\langle {A^{-1}x,x} \right\rangle \right)   g\left(
\left\langle {A^{-1}x,x} \right\rangle  \right) +
h^2\left(\left\langle {A^{-1}x,x} \right\rangle \right)f\left(
\left\langle {Ax,x} \right\rangle \right)   g\left( \left\langle
{Ax,x} \right\rangle \right)
\nonumber\\
&\ge (\le) h\left( \left\langle {Ax,x} \right\rangle
\right)h\left( \left\langle {A^{-1}x,x} \right\rangle \right)
\nonumber\\
&\qquad\times\left[f\left( \left\langle {A^{-1}x,x} \right\rangle
\right) g\left( \left\langle {Ax,x} \right\rangle \right) +
f\left( \left\langle {Ax,x} \right\rangle \right)  g\left(
\left\langle {A^{-1}x,x} \right\rangle \right)\right]
\label{eq2.10-3}
\end{align}
for any $x\in H$ with $\|x\|=1$.
\end{theorem}

\begin{proof}
Since $f, g$ are synchronous and $\gamma \le  \left\langle {Ax,x}
\right\rangle \le \Gamma$, $\gamma \le  \left\langle {By,y}
\right\rangle \le \Gamma$ for any $x,y\in H$ with $\|x\|=\|y\|=1$,
we have
\begin{multline}
\left( {h\left( \left\langle {Ax,x} \right\rangle \right)f\left(
\left\langle {By,y} \right\rangle \right) - h\left(\left\langle
{By,y} \right\rangle \right)f\left( \left\langle {Ax,x}
\right\rangle \right)} \right)
\\
\times\left( {h\left( \left\langle {Ax,x} \right\rangle
\right)g\left( \left\langle {By,y} \right\rangle  \right) -
h\left( \left\langle {By,y} \right\rangle  \right)g\left(
\left\langle {Ax,x} \right\rangle \right)} \right) \ge 0
\label{eq2.12-3}
\end{multline}
for any $t \in \left[a,b\right]$ for any $x\in H$ with $\|x\|=1$.

Employing property \eqref{eq1.2} for inequality \eqref{eq2.12-3}
we have
\begin{multline}
 h^2\left( \left\langle {Ax,x} \right\rangle \right)f\left(
\left\langle {By,y} \right\rangle \right)   g\left( \left\langle
{By,y} \right\rangle  \right)
\\
+  h^2\left(\left\langle {By,y} \right\rangle \right)f\left(
\left\langle {Ax,x} \right\rangle \right)  g\left( \left\langle
{Ax,x} \right\rangle \right)
\\
- h\left( \left\langle {Ax,x} \right\rangle \right)h\left(
\left\langle {By,y} \right\rangle \right)f\left( \left\langle
{By,y} \right\rangle \right)  g\left( \left\langle {Ax,x}
\right\rangle \right)
\\
- h\left(\left\langle {By,y} \right\rangle \right) h\left(
\left\langle {Ax,x} \right\rangle \right)f\left( \left\langle
{Ax,x} \right\rangle \right)  g\left( \left\langle {By,y}
\right\rangle  \right) \ge 0 \label{eq2.13-3}
\end{multline}
for any  bounded linear operator $B$ with $\spe\left({B}\right)
\subseteq \left[\gamma,\Gamma\right]$ and $y\in H$ with $\|y\|=1$.

Now, since
\begin{multline}
 h^2\left( \left\langle {Ax,x} \right\rangle \right)f\left(
\left\langle {By,y} \right\rangle \right) \cdot  g\left(
\left\langle {By,y} \right\rangle  \right) + h^2\left(\left\langle
{By,y} \right\rangle \right)f\left( \left\langle {Ax,x}
\right\rangle \right)\cdot  g\left( \left\langle {Ax,x}
\right\rangle \right)
\\
\ge h\left( \left\langle {Ax,x} \right\rangle \right)h\left(
\left\langle {By,y} \right\rangle \right) \left[f\left(
\left\langle {By,y} \right\rangle \right)  g\left( \left\langle
{Ax,x} \right\rangle \right) + f\left( \left\langle {Ax,x}
\right\rangle \right)  g\left( \left\langle {By,y} \right\rangle
\right)\right] \label{eq2.16-3}
\end{multline}
for each $x,y\in H$ with $\|x\|=\|y\|=1$. Setting $B=A^{-1}$ and
$y=x$ in \eqref{eq2.16-3} we get the required result in
\eqref{eq2.10-3}. The reverse sense follows similarly.
 \end{proof}

\begin{remark}
Let $0 < \gamma < \Gamma$ and choose $f(s)=s$ and $g(s)=s^{-1}$,
$s>0$ in Theorem \ref{thm2.2-3}. So that, if $f$ and $g$ are
$h$-synchronous ($h$-asynchronous) on $\left[ {\gamma,\Gamma}
\right]$, then
\begin{multline*}
 h^2\left( \left\langle {Ax,x} \right\rangle \right)
  + h^2\left(\left\langle {A^{-1}x,x} \right\rangle
\right)
\\
\ge (\le)2 h\left( \left\langle {Ax,x} \right\rangle
\right)h\left( \left\langle {A^{-1}x,x} \right\rangle \right)
 \left[
\left\langle {A^{-1}x,x} \right\rangle    \left\langle {Ax,x}
\right\rangle^{-1} +   \left\langle {Ax,x} \right\rangle
\left\langle {A^{-1}x,x} \right\rangle^{-1}
 \right]
\end{multline*}

 for any $x\in H$ with $\|x\|=1$.
 \end{remark}

 \begin{corollary}
\label{cor5-3}Let $A$ be a selfadjoint operator with
$\spe\left(A\right)\subset \left[\gamma,\Gamma\right]$  for some
real numbers $\gamma,\Gamma$ with $\gamma<\Gamma$. Let $h: \left[
{\gamma,\Gamma} \right]\to \mathbb{R}$ be a non-negative
continuous. If $f: \left[ {\gamma,\Gamma} \right]\to \mathbb{R}$
is continuous and  $h$-synchronous on $\left[ {\gamma,\Gamma}
\right]$, then
\begin{multline}
 h^2\left( \left\langle {Ax,x} \right\rangle \right)f^2\left(
\left\langle {A^{-1}x,x} \right\rangle \right)    +
h^2\left(\left\langle {A^{-1}x,x} \right\rangle \right)f^2\left(
\left\langle {Ax,x} \right\rangle \right)
\\
\ge  2 h\left( \left\langle {Ax,x} \right\rangle \right)h\left(
\left\langle {A^{-1}x,x} \right\rangle \right) f\left(
\left\langle {A^{-1}x,x} \right\rangle \right)  f\left(
\left\langle {Ax,x} \right\rangle \right)
  \label{eq2.17-3}
\end{multline}
 for any $x\in H$ with $\|x\|=1$.
\end{corollary}
\begin{proof}
Setting $f=g$ in \eqref{eq2.10-3}, respectively; we get the
required results.
\end{proof}

An $n$-operators version of Theorem \ref{thm2.1} is embodied as
follows:
 \begin{theorem}
\label{thm2.3}Let $A_j$ be a selfadjoint operator with
$\spe\left(A_j\right)\subset \left[\gamma,\Gamma\right]$  for
$j\in \{1,2,\cdots,n\}$  for some real numbers $\gamma,\Gamma$
with $\gamma<\Gamma$. Let $h: \left[ {\gamma,\Gamma} \right]\to
\mathbb{R}$ be a non-negative continuous. If $f,g: \left[
{\gamma,\Gamma} \right]\to \mathbb{R}$ are continuous and  both
 $h$-synchronous ($h$-asynchronous) on $\left[ {\gamma,\Gamma}
\right]$, then
\begin{multline}
\label{eq2.19}  \sum\limits_{j = 1}^n { \left\langle {h^2 \left(
A_j \right)x_j,x_j} \right\rangle} \sum\limits_{j = 1}^n
{\left\langle {f\left( A_j \right)g\left( A_j \right)x_j,x_j}
\right\rangle} \\\ge (\le) \sum\limits_{j = 1}^n {\left\langle
{h\left( A_j \right)g\left( A_j \right)x_j,x_j} \right\rangle}
\sum\limits_{j = 1}^n {\left\langle {h\left( A_j \right)f\left(
A_j \right)x_j,x_j} \right\rangle}
\end{multline}
for each $x_j\in H$, $j\in \{1,2,\cdots,n\}$ with $\sum\limits_{j
= 1}^n{\|x_j\|^2}=1$.
 \end{theorem}

\begin{proof}
As in (\cite{TF}, p.6), if we put
\begin{align*}
\widetilde{A}:=\left( {\begin{array}{*{20}c}
   {A_1 } &  \cdots  & 0  \\
    \vdots  &  \ddots  &  \vdots   \\
   0 &  \cdots  & {A_n }  \\
\end{array}} \right)
\end{align*}
and
\begin{align*}
\widetilde{x}:=\left( {\begin{array}{*{20}c}
   {x_1 }  \\
    \vdots   \\
   {x_n }  \\
\end{array}} \right)
\end{align*}
then we have $\spe\left(\widetilde{A}\right)\subset
\left[\gamma,\Gamma\right]$, $\|\widetilde{x}\|=1$, $\left\langle
{h^2 \left(\widetilde{ A} \right)\widetilde{x},\widetilde{x}}
\right\rangle= \sum\limits_{j = 1}^n { \left\langle {h^2 \left(
A_j \right)x_j,x_j} \right\rangle}$, $\left\langle {f\left(
\widetilde{A}\right)g\left(\widetilde{ A}
\right)\widetilde{x},\widetilde{x}} \right\rangle=\sum\limits_{j =
1}^n {\left\langle {f\left( A_j \right)g\left( A_j \right)x_j,x_j}
\right\rangle}$, \\

$\left\langle {h\left( \widetilde{A} \right)g\left( \widetilde{A}
\right)x,x} \right\rangle=\sum\limits_{j = 1}^n {\left\langle
{h\left( A_j \right)g\left( A_j \right)x_j,x_j} \right\rangle}$,

and $\left\langle {h\left( \widetilde{A} \right)f\left(
\widetilde{A} \right)\widetilde{x},\widetilde{x}}
\right\rangle=\sum\limits_{j = 1}^n {\left\langle {h\left( A_j
\right)f\left( A_j \right)x_j,x_j} \right\rangle}$. Applying
Theorem \ref{thm2.1} for $\widetilde{A}$ and $\widetilde{x}$ we
deduce the desired result.
\end{proof}

\begin{corollary}
\label{eq2.20}Let $A_j$ be a selfadjoint operator with
$\spe\left(A_j\right)\subset \left[\gamma,\Gamma\right]$  for
$j\in \{1,2,\cdots,n\}$   for some real numbers $\gamma,\Gamma$
with $\gamma<\Gamma$. Let $h: \left[ {\gamma,\Gamma} \right]\to
\mathbb{R}$ be a non-negative continuous. If $f: \left[
{\gamma,\Gamma} \right]\to \mathbb{R}$ is continuous and
 $h$-synchronous  on $\left[ {\gamma,\Gamma}
\right]$, then
\begin{align}
\left({\sum\limits_{j = 1}^n {\left\langle {h\left( A_j
\right)f\left( A_j \right)x_j,x_j} \right\rangle} }\right)^2 \le
\sum\limits_{j = 1}^n { \left\langle {h^2 \left( A_j
\right)x_j,x_j} \right\rangle} \sum\limits_{j = 1}^n {\left\langle
{f^2\left( A_j \right)x_j,x_j} \right\rangle} \label{eq2.20}
\end{align}
for each $x_j\in H$, $j\in \{1,2,\cdots,n\}$ with $\sum\limits_{j
= 1}^n{\|x_j\|^2}=1$.
 \end{corollary}

\begin{proof}
Setting $f=g$ in \eqref{eq2.19}, we get the desired result.
\end{proof}

\begin{corollary}
Let $A_j$ be a selfadjoint operator with
$\spe\left(A_j\right)\subset \left[\gamma,\Gamma\right]$  for
$j\in \{1,2,\cdots,n\}$  for some real numbers $\gamma,\Gamma$
with $0<\gamma<\Gamma$.  If $f: \left[ {\gamma,\Gamma} \right]\to
\mathbb{R}$ is continuous and
 $t$-synchronous   on $\left[ {\gamma,\Gamma}
\right]$, then
\begin{align}
\left({\sum\limits_{j = 1}^n {\left\langle {  A_j  f\left( A_j
\right)x_j,x_j} \right\rangle} }\right)^2 \le  \sum\limits_{j =
1}^n { \left\langle {A^2_j x_j,x_j} \right\rangle} \sum\limits_{j
= 1}^n {\left\langle {f^2\left( A_j \right)x_j,x_j} \right\rangle}
\label{eq2.21}
\end{align}
for each $x_j\in H$, $j\in \{1,2,\cdots,n\}$ with $\sum\limits_{j
= 1}^n{\|x_j\|^2}=1$.
 \end{corollary}

\begin{proof}
Setting $h(t)=t$ in \eqref{eq2.20}, we get the desired result.
\end{proof}

\begin{remark}
Let $A_j$ be a selfadjoint operator with
$\spe\left(A_j\right)\subset \left[\gamma,\Gamma\right]$  for
$j\in \{1,2,\cdots,n\}$  for some real numbers $\gamma,\Gamma$
with $0<\gamma<\Gamma$.
 Let  $f\left( {s} \right)=s^{1/2}$  for $s\in \left[\gamma,\Gamma\right]$
 then $f$ is $t^{-1/2}$-synchronous so that by \eqref{eq2.21} we have
 \begin{align}
 n^2 \le
\left(\sum\limits_{j = 1}^n { \left\langle {A_jx_j,x_j}
\right\rangle}\right)\cdot \left(\sum\limits_{j = 1}^n
{\left\langle {A^{-1}_jx_j,x_j} \right\rangle}\right).
\label{eq.Cheb}
\end{align}
The discrete version of Chebyshev inequality, reads that
\begin{align*}
\frac{1}{m}\sum\limits_{i = 1}^m {a_i b_i }  \ge \left(
{\frac{1}{m}\sum\limits_{i = 1}^m {a_i } } \right)\left(
{\frac{1}{m}\sum\limits_{i = 1}^m {b_i } } \right)
\end{align*}
for all similarly ordered $n$-tuples $\left(a_1,\cdots,a_m\right)$
and $\left(b_1,\cdots,b_m\right)$.

Let $\{A_j\}_{j=1}^n$ be a finite positive sequence of invertible
self-adjoint operators and consider $a_j=\left\langle {A_jx_j,x_j}
\right\rangle$ and $b_j=\left\langle {A^{-1}_jx_j,x_j}
\right\rangle$ for all $j=1,\cdots,n$. If
$\left(a_1,\cdots,a_n\right)$ and $\left(b_1,\cdots,b_n\right)$
similarly ordered $n$-tuples. Then by employing the Chebyshev
inequality on \eqref{eq.Cheb} we get
 \begin{align}
1 \le \frac{1}{n} \sum\limits_{j = 1}^n { \left\langle {A
_jx_j,x_j} \right\rangle} \cdot \frac{1}{n}\sum\limits_{j = 1}^n
{\left\langle {A^{-1}_jx_j,x_j} \right\rangle} \le
\frac{1}{n}\sum\limits_{j = 1}^n {\left\langle {A_jx_j,x_j}
\right\rangle\left\langle {A^{-1}_jx_j,x_j}
\right\rangle}.\label{eq.cheb.Kant}
\end{align}
On other hand, if $\gamma_j \cdot 1_H\le A_j\le \Gamma_j \cdot
1_H$, then by Kanotrovich inequality we have
 \begin{align}
1 &\le \frac{1}{n} \sum\limits_{j = 1}^n { \left\langle {A
_jx_j,x_j} \right\rangle} \cdot \frac{1}{n}\sum\limits_{j = 1}^n
{\left\langle {A^{-1}_jx_j,x_j} \right\rangle}
\nonumber\\
 &\le \frac{1}{n}\sum\limits_{j = 1}^n {\left\langle {A_jx_j,x_j}
\right\rangle\left\langle {A^{-1}_jx_j,x_j} \right\rangle}  \le
\frac{1}{n}\sum\limits_{j = 1}^n {\frac{{\left( {\Gamma _j -
\gamma _j } \right)^2 }}{{4\gamma _j \Gamma _j }}}.
\label{refin.kant}
\end{align}
In case $n=2$, we have
 \begin{align*}
1 &\le \frac{1}{4} \left[{\left\langle {A_1x_1,x_1}
\right\rangle+\left\langle {A_2x_2,x_2} \right\rangle}\right]
\cdot  \left[{\left\langle {A^{-1}_1x_1,x_1}
\right\rangle+\left\langle {A^{-1}_2x_2,x_2} \right\rangle}\right]
\\
 &\le \frac{1}{2}\left[{\left\langle {A_1x_1,x_1}
\right\rangle\left\langle {A^{-1}_1x_1,x_1}
\right\rangle+\left\langle {A_2x_2,x_2} \right\rangle\left\langle
{A^{-1}_2x_2,x_2} \right\rangle}\right]
\\
&\le \frac{1}{8}\left[{\frac{{\left( {\Gamma _1 - \gamma _1 }
\right)^2 }}{{ \gamma _1 \Gamma _1 }}+\frac{{\left( {\Gamma _2 -
\gamma _2 } \right)^2 }}{{ \gamma _2 \Gamma _2 }}}\right].
\end{align*}

\end{remark}

An $n$-operators version of Theorem \ref{thm2.2} is incorporated
in the following result.
\begin{theorem}
\label{thm2.4}Let $A_j$ be a selfadjoint operator with
$\spe\left(A_j\right)\subset \left[\gamma,\Gamma\right]$  for
$j\in \{1,2,\cdots,n\}$  for some real numbers $\gamma,\Gamma$
with $\gamma<\Gamma$.  Let $h: \left[ {\gamma,\Gamma} \right]\to
\mathbb{R}$ be a non-negative continuous.
 If $f,g: \left[ {\gamma,\Gamma} \right]\to \mathbb{R}$ are
continuous and both $f$ and $g$ are $h$-synchronous
($h$-asynchronous ) on $\left[ {\gamma,\Gamma} \right]$, then
\begin{align}
& h^2 \left( {\sum\limits_{j = 1}^n{\left\langle {A_jx_j,x_j}
\right\rangle }} \right)
 \sum\limits_{j = 1}^n {\left\langle {f\left( A_j
\right)g\left( A_j \right)x_j,x_j} \right\rangle}
\nonumber\\
&\qquad- \sum\limits_{j = 1}^n {\left\langle {h\left( A_j
\right)f\left( A_j \right)x_j,x_j} \right\rangle}   \sum\limits_{j
= 1}^n {\left\langle {h\left( A_j \right)g\left( A_j
\right)x_j,x_j} \right\rangle}
\nonumber\\
& \ge (\le) \left[{h\left( {\sum\limits_{j = 1}^n {\left\langle
{A_jx_j,x_j} \right\rangle} } \right) \sum\limits_{j = 1}^n
{\left\langle {h\left( A_j \right)f\left( A_j \right)x_j,x_j}
\right\rangle} }\right.\label{eq2.22}
\\
&\qquad \qquad\left.{-  \sum\limits_{j = 1}^n { \left\langle {h^2
\left( A_j \right)x_j,x_j} \right\rangle} f\left( {\sum\limits_{j
= 1}^n {\left\langle {A_jx_j,x_j} \right\rangle} } \right)
}\right] \cdot g\left( {\sum\limits_{j= 1}^n {\left\langle
{A_jx_j,x_j} \right\rangle }} \right)
 \nonumber\\
&\qquad + \left[{h\left( {\sum\limits_{j = 1}^n {\left\langle
{A_jx_j,x_j} \right\rangle }}\right)f\left( {\sum\limits_{j = 1}^n
{\left\langle {A_jx_j,x_j} \right\rangle} } \right) }\right.
\nonumber\\
&\qquad\qquad\left.{- \sum\limits_{j = 1}^n {\left\langle {h\left(
A_j \right)f\left( A_j \right)x_j,x_j} \right\rangle} }\right]
\cdot \sum\limits_{j = 1}^n {\left\langle {h\left( A_j
\right)g\left( A_j \right)x_j,x_j} \right\rangle }\nonumber
\end{align}
for each $x_j\in H$, $j\in \{1,2,\cdots,n\}$ with $\sum\limits_{j
= 1}^n{\|x_j\|^2}=1$.

\end{theorem}

\begin{proof}
The proof is similar to the proof of Theorem \ref{thm2.3} on
employing Theorem \ref{thm2.2}.
\end{proof}

\begin{corollary}
 Let $A_j$ be a selfadjoint operator with
$\spe\left(A_j\right)\subset \left[\gamma,\Gamma\right]$  for
$j\in \{1,2,\cdots,n\}$   for some real numbers $\gamma,\Gamma$
with $\gamma<\Gamma$.  Let $h: \left[ {\gamma,\Gamma} \right]\to
\mathbb{R}$ be a non-negative continuous and convex on $\left[
{\gamma,\Gamma} \right]$.
 If $f: \left[ {\gamma,\Gamma} \right]\to \mathbb{R}$ is
continuous and  $h$-synchronous   on $\left[ {\gamma,\Gamma}
\right]$, then
\begin{align}
&  h^2 \left( {\sum\limits_{j = 1}^n{\left\langle {A_jx_j,x_j}
\right\rangle }} \right)\sum\limits_{j = 1}^n {\left\langle
{f^2\left( A_j \right) x_j,x_j} \right\rangle}
 - \left(\sum\limits_{j = 1}^n {\left\langle {h\left( A_j
\right)f\left( A_j \right)x_j,x_j} \right\rangle} \right)^2
\nonumber\\
& \ge (\le) \left[{h\left( {\sum\limits_{j = 1}^n {\left\langle
{A_jx_j,x_j} \right\rangle} } \right) \sum\limits_{j = 1}^n
{\left\langle {h\left( A_j \right)f\left( A_j \right)x_j,x_j}
\right\rangle} }\right.\label{eq2.24}
\\
&\qquad \qquad\left.{-  \sum\limits_{j = 1}^n { \left\langle {h^2
\left( A_j \right)x_j,x_j} \right\rangle} f\left( {\sum\limits_{j
= 1}^n {\left\langle {A_jx_j,x_j} \right\rangle} } \right)
}\right] \cdot f\left( {\sum\limits_{j= 1}^n {\left\langle
{A_jx_j,x_j} \right\rangle }} \right)
 \nonumber\\
&\qquad + \left[{h\left( {\sum\limits_{j = 1}^n {\left\langle
{A_jx_j,x_j} \right\rangle }}\right)f\left( {\sum\limits_{j = 1}^n
{\left\langle {A_jx_j,x_j} \right\rangle} } \right) }\right.
\nonumber\\
&\qquad\qquad\left.{- \sum\limits_{j = 1}^n {\left\langle {h\left(
A_j \right)f\left( A_j \right)x_j,x_j} \right\rangle} }\right]
\cdot \sum\limits_{j = 1}^n {\left\langle {h\left( A_j
\right)f\left( A_j \right)x_j,x_j} \right\rangle }\nonumber
\end{align}

for each $x_j\in H$, $j\in \{1,2,\cdots,n\}$ with $\sum\limits_{j
= 1}^n{\|x_j\|^2}=1$.

\end{corollary}

\begin{proof}
Setting $f=g$ in \eqref{eq2.22}, respectively; we get the required
results.
\end{proof}

\begin{corollary}
 Let $A_j$ be a selfadjoint operator with
$\spe\left(A_j\right)\subset \left[\gamma,\Gamma\right]$  for
$j\in \{1,2,\cdots,n\}$   for some real numbers $\gamma,\Gamma$
with $0<\gamma<\Gamma$.
 If $f: \left[ {\gamma,\Gamma} \right]\to \mathbb{R}$ is
continuous and  $t$-synchronous  on $\left[ {\gamma,\Gamma}
\right]$, then

\begin{align}
&   \left( {\sum\limits_{j = 1}^n{\left\langle {A_jx_j,x_j}
\right\rangle }} \right)^2\sum\limits_{j = 1}^n {\left\langle
{f^2\left( A_j \right) x_j,x_j} \right\rangle}
 - \left(\sum\limits_{j = 1}^n {\left\langle {A_jf\left( A_j \right)x_j,x_j} \right\rangle} \right)^2
\nonumber\\
& \ge (\le) \left[{ \sum\limits_{j = 1}^n {\left\langle
{A_jx_j,x_j} \right\rangle}   \sum\limits_{j = 1}^n {\left\langle
{A_jf\left( A_j \right)x_j,x_j} \right\rangle}
}\right.\label{eq2.26}
\\
&\qquad \qquad\left.{-  \sum\limits_{j = 1}^n { \left\langle
{A^2_jx_j,x_j} \right\rangle} f\left( {\sum\limits_{j = 1}^n
{\left\langle {A_jx_j,x_j} \right\rangle} } \right) }\right] \cdot
f\left( {\sum\limits_{j= 1}^n {\left\langle {A_jx_j,x_j}
\right\rangle }} \right)
 \nonumber\\
&\qquad + \left[{  \sum\limits_{j = 1}^n {\left\langle
{A_jx_j,x_j} \right\rangle } f\left( {\sum\limits_{j = 1}^n
{\left\langle {A_jx_j,x_j} \right\rangle} } \right) }\right.
\nonumber\\
&\qquad\qquad\left.{- \sum\limits_{j = 1}^n {\left\langle
{A_jf\left( A_j \right)x_j,x_j} \right\rangle} }\right] \cdot
\sum\limits_{j = 1}^n {\left\langle {A_jf\left( A_j
\right)x_j,x_j} \right\rangle }\nonumber
\end{align}

for each $x_j\in H$, $j\in \{1,2,\cdots,n\}$ with $\sum\limits_{j
= 1}^n{\|x_j\|^2}=1$.

\end{corollary}
\begin{proof}
Setting $h(t)=t$ in \eqref{eq2.24}, we get the desired results.
\end{proof}

\begin{remark}
 By choosing $h\left(t\right)=1$  for all $t\in
\left[a,b\right]$, in Theorems \ref{thm2.1}, \ref{thm2.2},
\ref{thm2.3} and \ref{thm2.4}, then we recapture all inequalities
obtained \cite{SD1} and their consequences.
\end{remark}

\noindent \textbf{Acknowledgment.} The author  wish to thank the
 referees for their careful reading and for providing very
 constructive
comments that helped improving the presentation of this article.


\bibliographystyle{amsplain}

\end{document}